\documentclass[12pt,reqno]{amsart}

\usepackage[T1]{fontenc}
\usepackage[utf8]{inputenc}

\usepackage{graphicx}
\usepackage[dvipsnames]{xcolor}

\usepackage{amsmath, amscd, amsthm,mathtools}

\usepackage{libertine}          
\usepackage[libertine]{newtxmath} 
\usepackage[scr=rsfs, cal=euler, bb=ams]{mathalfa} 

\usepackage{bm} 

\usepackage[margin=2.5cm]{geometry} 

\usepackage{verbatim, comment}
\usepackage[numbers]{natbib}

\everymath{\displaystyle}

\usepackage[
    colorlinks=true,
    allcolors=NavyBlue,      
    bookmarksnumbered=true,
    psdextra,             
    pdfencoding=auto
]{hyperref}

\usepackage[nameinlink, capitalize]{cleveref}

\usepackage{aliascnt}

\theoremstyle{plain}

\newtheorem{thm}{Theorem}[section]
\crefname{thm}{Theorem}{Theorems}

\newaliascnt{lem}{thm}
\newtheorem{lem}[lem]{Lemma}
\aliascntresetthe{lem}
\crefname{lem}{Lemma}{Lemmas}

\newaliascnt{prop}{thm}
\newtheorem{prop}[prop]{Proposition}
\aliascntresetthe{prop}
\crefname{prop}{Proposition}{Propositions}

\newaliascnt{cor}{thm}

\aliascntresetthe{cor}
\crefname{cor}{Corollary}{Corollaries}

\newaliascnt{fact}{thm}

\aliascntresetthe{fact}
\crefname{fact}{Fact}{Facts}

\newaliascnt{conj}{thm}

\aliascntresetthe{conj}
\crefname{conj}{Conjecture}{Conjectures}

\newtheorem{mainthm}{Main Theorem}

\newaliascnt{mainconj}{mainthm}

\aliascntresetthe{mainconj}
\crefname{mainconj}{Main Conjecture}{Main Conjectures}

\theoremstyle{definition}

\newaliascnt{dfn}{thm}

\aliascntresetthe{dfn}
\crefname{dfn}{Definition}{Definitions}

\newaliascnt{rem}{thm}
\newtheorem{rem}[rem]{Remark}
\aliascntresetthe{rem}
\crefname{rem}{Remark}{Remarks}

\newaliascnt{ex}{thm}

\aliascntresetthe{ex}
\crefname{ex}{Example}{Examples}

\newcommand{\NN}{\mathbb{N}}
\newcommand{\ZZ}{\mathbb{Z}}
\newcommand{\QQ}{\mathbb{Q}}
\newcommand{\QQbar}{\overline{\mathbb{Q}}}

\newcommand{\CC}{\mathbb{C}}
\newcommand{\A}{\mathcal{A}}
\newcommand{\C}{\mathcal{C}}
\newcommand{\PP}{\mathcal{P}}
\newcommand{\OO}{\mathcal{O}}
\newcommand{\F}{\mathbb{F}}
\newcommand{\Gal}{\mathrm{Gal}}
\newcommand{\Frob}{\mathrm{Frob}}
\newcommand{\Spec}{\mathrm{Spec}}
\newcommand{\GL}{\mathrm{GL}}
\newcommand{\Aut}{\mathrm{Aut} }
\newcommand{\tr}{\mathrm{tr}}
\newcommand{\fp}{\mathfrak{p}}
\newcommand{\Kbar}{\overline{K}}
\newcommand{\GSp}{\mathrm{GSp}}
\newcommand{\End}{\mathrm{End}}
\newcommand{\Ker}{\mathrm{Ker}}

\title[$\PP_{\A}^{0}$-transcendence of the Frobenius traces]{On the finite transcendence of Frobenius traces for abelian varieties over $\QQ$}

\author{Yuto Tsuruta}
\address{Mathematical Institute, Tohoku University, Sendai 980-8578, Japan}
\email{tsuruta.yuuto.q7@dc.tohoku.ac.jp}

\date{}

\begin{document}

\begin{abstract}
    The first purpose of this paper is to give the finite transcendence of Frobenius traces for elliptic curves over $\QQ$ without the assumption of complex multiplication (CM). This result generalizes the previous work by Luca and Zudilin, who obtained similar transcendence results specifically for the CM case. The second purpose is to give the finite transcendence of Frobenius traces for a certain class of abelian varieties over $\QQ$, by using Galois representations and Luca--Zudilin's method.
\end{abstract}

\maketitle

\section{Introduction and Main Theorem}
The \textit{finite algebraic number} which lives in the ring
\begin{align}
\nonumber
    \A=\left. \prod_{p:\text{prime}}\F_{p}\middle/\bigoplus_{p:\text{prime}}\F_{p} \right.
\end{align}
was first introduced by Rosen (\cite{Rosen}) in 2020. Rosen constructed analogues of $\QQbar$ and of (complex) periods in $\A$, and investigated their properties including their algebraic structures. The set of finite algebraic numbers defined in \cite{Rosen} is written as $\PP_{\A}^{0}$. It can be characterized as follows:
\begin{thm}[{\cite[Theorem 1.1]{Rosen}}]\label{Thm: Rosen}
  Let $\bm{t}=(t_p\bmod p)_p\in\A$. Then, the following are equivalent:
  \begin{enumerate}
    \item \label{item:algebraic1} $\bm{t} \in \PP_{\A}^{0}$.
    \item \label{item:algebraic2} There exists a linear recurrent sequence $(a_{n})_{n}\in\QQ^{\NN}$ such that $\bm{t}=(a_{p}\bmod p)_{p}$.
    \item \label{item:algebraic3} There exists a finite Galois extension $L/\QQ$ and a map $\phi \colon \Gal(L/\mathbb{Q}) \to L$ satisfying
    \begin{align}
    \nonumber
        \phi(\sigma\tau\sigma^{-1}) = \sigma \phi(\tau) \quad \text{for all } \sigma, \tau \in \Gal(L/\QQ),
    \end{align}
    such that
    \begin{align}
    \nonumber
        \bm{t} = \left( \phi(\Frob_{\fp}) \bmod \fp \right)_{p}.
    \end{align}
    Here, $\fp$ is a prime ideal of $\OO_L$ lying over $p$ and $\Frob_{\fp}$ is a Frobenius element in the decomposition group $D_{\fp} \subset \Gal(L/\QQ)$.
\end{enumerate}
\end{thm}

\begin{rem}[cf. {\cite[Section 2]{Rosen}}]
    To avoid possible misunderstanding concerning the notation used in \Cref{Thm: Rosen}-(\ref{item:algebraic3}), we clarify the following. Let $p$ be a rational prime and $\fp$ be a prime ideal of $\OO_L$ lying over $p$ such that the extension $L/\QQ$ is unramified at $\fp$. Assume that $\fp$ is coprime to the denominators of all values of $\phi$. By definition of $\phi\colon\Gal(L/\QQ)\to L$, it follows that the residue class $\phi(\Frob_{\fp})\bmod\fp$ is fixed by $\Frob_{\fp}$. Therefore, we see that
    \begin{align}
    \nonumber
    \phi(\Frob_{\fp})\bmod\fp\in\F_{p}\subset\OO_{L}/\fp.
    \end{align}
    In addition, we note that the residue class $\phi(\Frob_{\fp})\bmod\fp$ is independent of the choice of $\fp\mid p$ (see {\cite[Section 4]{Rosen2}}). Thus, $(\phi(\Frob_\fp)\bmod\fp)_p$ is well-defined as an element of $\A$. 
\end{rem}
The above theorem seems to differ from the classical definition of algebraic numbers. However, Rosen pointed out that every element in $\PP_{\A}^{0}$ is a root of a non-zero polynomial:
\begin{thm}[{\cite[Theorem 1.4]{Rosen}}]\label{Thm: Rosen 1.4}
    If the element $\bm{t}=(t_{p}\bmod p)_{p}$ is in $\PP_{\A}^{0}$, then, there exists a non-zero polynomial $f(X)\in\QQ[X]$ such that $f(\bm{t})=0$ in $\A$, and every such $f(X)$ has a rational root.
\end{thm}
The key point of \Cref{Thm: Rosen 1.4} is the assertion that every such polynomial has a rational root. Thus, one might instead regard the set 
\begin{align}
\nonumber
    \C_{\A}=\{\bm{t}\in\A\ |\ {}^{\exists}f(X)\in\QQ[X]\setminus\{0\}\ \text{s.t.}\ f(\bm{t})=0\ \text{in}\ \A\}
\end{align}
as a more natural candidate for the set of finite algebraic numbers. However, it is known that $\C_{\A}$ fails to have the fundamental properties of $\QQbar$ as the following proposition shows.
\begin{prop}[{\cite[Section 4.1]{Rosen}}, {\cite[Proposition 3.8]{AF}}]
The following claims hold.
    \begin{enumerate}
        \item $\C_{\A}$ is an uncountable set.
        \item There are natural strict inclusions $\QQ\subsetneq\PP_{\A}^{0}\subsetneq\C_{\A}\subsetneq\A$.
    \end{enumerate}
\end{prop}
Furthermore, it is also known that the elements of $\PP_{\A}^{0}$ are obtained via the \textit{$\A$-valued Frobenius automorphism} 
\begin{align}
\nonumber
   F_{\A}\colon H^{0}_{dR}(\Spec(L))\otimes_{\QQ}\A\xrightarrow{\sim} H^{0}_{dR}(\Spec(L))\otimes_{\QQ}\A.
\end{align}
Here, $H^{0}_{dR}(\Spec(L))=L$ is the 0-th algebraic de Rham cohomology of 
$\Spec(L)$ and $L$ runs over the number fields (cf. {\cite[Theorem 4.2]{Rosen}}). Therefore, $\PP_{\A}^{0}$ admits some geometric interpretations (the ``$0$'' of $\PP_{\A}^{0}$ comes from the ``$0$''-th cohomology).

It is also an interesting problem to give an element of $\A$ which is not contained in $\PP_{\A}^0$. We call those elements \textit{$\PP_{\A}^{0}$-transcendental numbers}. The first attempt at finding such elements has been given by Luca--Zudilin.

\begin{thm}[{\cite[Theorem 1.3]{LZ2}}]\label{Thm: q-Fib}
    Define the $q$-Fibonacci sequence $(F_{n}(q))_{n}\in\left(\ZZ[q]\right)^{\NN}$ by $F_{0}(q)=0,\ F_{1}(q)=1$ and
    \begin{align}
        \nonumber
        F_{n}(q)=F_{n-1}(q)+q^{n-2}F_{n-2}(q),\quad {}^{\forall}n\in\ZZ_{\geq2}.
    \end{align}
    Then, $(F_{p}(q)\bmod p)_{p}$ is $\PP_{\A}^{0}$-transcendental for all $q\in\ZZ_{>1}$.
\end{thm}

\begin{rem}
    Anzawa--Funakura's result (\cite[Theorem 1.2]{AF}) is also important, as they were the first to obtain the result about \Cref{Thm: q-Fib}. They proved the ``$\C_{\A}$-transcendence'' of $(F_{p}(q)\bmod p)_{p}$ for all square-free $q>1$, under the generalized Riemann hypothesis, earlier than Luca--Zudilin (\cite{LZ2}). Here, we note that the two results should be understood separately, as they differ in both their conditions, conclusions, and methods of proof.

    For $\C_{\A}$-transcendence, we refer to the paper by Matsusaka and Seki (\cite{MS}). The authors succeeded in proving that $(F_{p}(q)\bmod p)_{p}$ is $\C_{\A}$-transcendental for all $q\in\ZZ_{>1}$. It is also a generalization of both results in \cite{AF,LZ2}. They also carried out other investigations about $\C_{\A}$-transcendence in \cite{MS}.
\end{rem}

After \Cref{Thm: q-Fib}, Luca and Zudilin obtained new $\PP_{\A}^{0}$-transcendental elements by using elliptic curves:

\begin{thm}[{\cite[Theorem 1]{LZ1}}]\label{Thm: LZ}
  Let $E$ be an elliptic curve over $\QQ$ without complex multiplication. For any prime $p$, we define
  \begin{align}
  \nonumber
      a_p(E)=\begin{cases}
           p+1-\# E(\F_p)  & \text{if}\ E\ \text{has good reduction at}\ p,\\ 
           0 & \text{otherwise}.
      \end{cases}
  \end{align}
  Then, the sequence of Frobenius traces $\bm{a}(E)=(a_{p}(E)\bmod p)_{p}$ is $\PP_{\A}^{0}$-transcendental.
\end{thm}

Our first main result is an extension of Luca--Zudilin's result, which removes the condition ``complex multiplication''.

\begin{mainthm}\label{Main 1}
  Let $E$ be an elliptic curve over $\QQ$. Then, the sequence of Frobenius traces $\bm{a}(E)=(a_{p}(E)\bmod p)_{p}$ is $\PP_{\A}^{0}$-transcendental. 
\end{mainthm}

\begin{rem}
    Although \Cref{Main 1} overlaps with \cite[Theorem 4.1]{MS} and both results require the Sato--Tate conjecture to handle the CM case, these proofs are different from each other.
   
    In \cite{MS}, they proved the $\C_\A$-transcendence of $\bm{a}(E)$ by formulating and applying a criterion for $\C_\A$-transcendence based on the proof in \cite[Theorem 1]{LZ1}. They referred to this criterion as the \textit{Luca--Zudilin criterion} (cf. {\cite[Lemma 2.5]{MS}}). Although elementary, it serves as a highly effective criterion with various applications.

    On the other hand, in the present paper, we focus only on rational primes that split completely in certain finite Galois extensions over $\QQ$ and use the Chebotarev density theorem to deduce a contradiction. In addition, our proof employs classical theorems on elliptic curves, such as \textit{Deuring's reduction theorem}. Consequently, these two approaches are conceptually distinct: while the proof in \cite{MS} relies on a general, elementary criterion for sequences, our argument is inherently grounded in the algebraic and geometric properties of elliptic curves.
\end{rem}
By using techniques in the proof of \Cref{Main 1}, we give certain $\PP_{\A}^{0}$-transcendental elements as follows:

\begin{mainthm}\label{Main 2}
   Let $A$ be an abelian variety over $\QQ$. Assume that $A$ satisfies the conditions in \Cref{Thm: Serre open image}; that is, the $\ell$-adic Galois representation
     \begin{align}
    \nonumber
    \rho_{A,\ell}\colon G_\QQ\rightarrow \GSp_{2g}(\ZZ_{\ell})
    \end{align}
    is surjective for all but finitely many primes $\ell$. For any prime $p$, we define
    \begin{align}
    \nonumber
        a_p(A)=\begin{cases}
            \tr(\rho_{A,\ell}(\Frob_{p})) & \text{if}\ A\ \text{has good reduction at}\ p,\\
            0 & \text{otherwise.}
        \end{cases}
    \end{align}
    Then, the sequence of Frobenius traces $\bm{a}(A)=(a_{p}(A)\bmod p)_{p}$ is $\PP_{\A}^{0}$-transcendental.
\end{mainthm}

\begin{rem}\label{remark: good reduction Galois representation}
    In general, $A$ has good reduction for all but finitely many $p$ by Serre and Tate (\cite[Theorem 1]{ST}). It is known that the characteristic polynomial of Frobenius has integral coefficients and does not depend on the choice of $p$ (cf. \cite{Del}). Thus, $\bm{a}(A)$ is well-defined as an element of $\A$.
\end{rem}

\subsection*{Notation of this paper}
Throughout this paper, we use the following standard notation.
\begin{itemize}
    \item $p,\ \ell$ denote distinct prime numbers. 
    \item For a number field $L$, we write its ring of integers as $\OO_L$.
    \item For a number field $K$, $\Kbar$ denotes a fixed algebraic closure of $K$.
    \item For a number field $K$, we denote the absolute Galois group $\Gal(\Kbar/K)$ by $G_K$.  
    \item For a subset of primes $S$, we write the natural density as $\text{den}(S)=\lim_{X\to\infty}\frac{\#\{ p\leq X\ |\ p\in S \}}{\#\{ p\leq X\}}$.
\end{itemize}

To state the asymptotic behavior of functions, we use the symbols $O$, $\ll$ $o$, and $\sim$ defined as follows.
\begin{itemize}
    \item $f(x) = O(g(x))$ if there exists a positive constant $C$ such that $|f(x)| \leq C g(x)$ for sufficiently large $x$. We also use the symbol $f(x) \gg g(x)$, which is equivalent to $f(x) = O(g(x))$ and $g(x) = O(f(x))$.
    \item $f(x) = o(g(x))$ if $f(x)/g(x)\to0\ (x \to \infty)$.
    \item $f(x) \sim g(x)$ if $f(x)/g(x)\to1\ (x \to \infty)$.
\end{itemize}
When the above notation depends on some parameters $\epsilon$, we may use this dependence by adding subscripts, such as $\ll_{\epsilon}$.

\section*{Acknowledgements}
This work originated from an informal private seminar held under the kind guidance of Professor Takuya Yamauchi at Tohoku University. The author would like to express his deepest gratitude to Professor Yamauchi for providing the opportunity to discuss this work and for his fruitful advice and insightful remarks.

The author is also deeply grateful to Professors Toshiki Matsusaka, Shin-ichiro Seki, Yasuo Ohno, and Wadim Zudilin for their careful reading of the present paper and for many helpful suggestions. The author would like to thank Mahiro Yokomizo for valuable comments on a previous version of this paper.

Finally, the author would like to thank Tohoku University for its incredible hospitality during the course of this research project.

\section{Transcendence for the Frobenius traces of elliptic curves}\label{section: fin. trans. ell. case}
In \cite{LZ1}, Luca and Zudilin proved \Cref{Thm: LZ} by using \textit{Serre's open image theorem}, stated as follows:

\begin{thm}[{\cite[Section 4.4]{Serre}}]
For any non-CM elliptic curve $E$ over $\QQ$, the $\ell$-adic Galois representation 
\begin{align}
\nonumber
\rho_{E,\ell}\colon G_\QQ\to\Aut_{\ZZ_\ell}(T_\ell E)\simeq\GL_{2}(\ZZ_{\ell})
\end{align}
is surjective for all but finitely many primes $\ell$. Here, $T_\ell E$ is the $\ell$-adic Tate module of $E$. 
\end{thm}

Our proof proceeds without using the above theorem, focusing instead on primes that split completely in certain finite Galois extensions over $\QQ$ and applying the \textit{Sato--Tate conjecture}. Here we summarize the Sato--Tate conjecture. The CM case is due to Hecke and we refer to \cite{Murty}. The non-CM case is due to a celebrated work \cite{BGHT}. For an elliptic curve $E/\QQ$ and a prime $p$ such that $E$ has good reduction at $p$, we choose $\theta_p(E)\in [0,\pi]$ satisfying $\cos{\theta_{p}(E)}:=a_{p}(E)/2\sqrt{p}$. We call such a prime $p$ a good prime for $E$. 

\begin{thm}[Sato--Tate conjecture, \citep{BGHT, Murty}]\label{Thm: Sato--Tate conj}
Let $E$ be an elliptic curve over $\QQ$. 
Then, the following holds.
\begin{enumerate}
\item \label{item:Sato--Tate nonCM} When $E$ is non-CM, then for $0\leq \alpha<\beta\leq\pi$, it holds
\begin{align}
\nonumber
  \lim_{X\to\infty}\frac{\#\left\{ p\leq X\ |\ \text{$p$ is a good prime for $E$},\ \alpha\leq \theta_{p}(E)\leq\beta\right\}}{\#\left\{ p\ |\ p\leq X \right\}}=\frac{2}{\pi}\int_{\alpha}^{\beta}\sin^{2}{\theta}d\theta.
\end{align}
\nonumber
\item \label{item:Sato--Tate CM} When $E$ has CM, then for $0\leq \alpha<\beta\leq\pi$, it holds
\begin{align}
  \lim_{X\to\infty}\frac{\#\left\{ p\leq X\ |\ \text{$p$ is a good prime for $E$},\ \alpha\leq \theta_{p}(E)\leq\beta\right\}}{\#\left\{ p\ |\ p\leq X \right\}}=\frac{\delta_{0}}{2}+\frac{1}{2\pi}\int_{\alpha}^{\beta}d\theta.
\end{align}
Here, $\delta_{0}=\begin{cases}
    1,\ \mathrm{if}\ \pi/2\in[\alpha,\beta],\\ 0,\ \mathrm{otherwise}.
\end{cases}$
\end{enumerate}
\end{thm} 

Before starting the proof of \Cref{Main 1}, we state the following lemma.
\begin{lem}\label{lem: existence of primes}
    Let $E$ be an elliptic curve over $\QQ$ with complex multiplication by an imaginary quadratic field $K$. For any number field $L$, define the set of primes
    \begin{align}
    \nonumber
        S=\{ p\colon \text{good prime for}\ E\ |\ p\ \text{splits completely in}\ L\}.
    \end{align}
    Then, there exist infinitely many primes $p$ in $S$ such that $a_{p}(E)\not\equiv0 \pmod p$. 
\end{lem}

\begin{proof}
By Deuring's reduction theorem (cf. {\cite[Theorem 1.1]{Zaytsev}}), a rational prime $p$ splits completely in $K$ if and only if $a_{p}(E)\not\equiv0 \pmod p$. We define two sets
    \begin{align}
    \nonumber
        & T=\{ p\colon \text{good prime for}\ E\ |\ p\ \text{splits completely in}\ K\},\\
    \nonumber
        & T^{\prime}=\{ p\colon \text{good prime for}\ E\ |\ p\ \text{splits completely in}\ KL\}.
    \end{align}
    By definition, we see that $T^{\prime}\subseteq T$ and $T\subseteq S$. Applying the Chebotarev density theorem (cf. {\cite[Th\'eor\`eme 1, Section 2.1]{Serre Chebotarev}}), we see that $\text{den}(T)\geq \text{den}(T^{\prime})>0$. Thus,we obtain the conclusion.     
\end{proof}

\begin{proof}[Proof of \Cref{Main 1}]
We assume $\bm{a}(E)=(a_{p}(E)\bmod p)_{p}\in\PP_{\A}^{0}$ (and consequently, deduce a contradiction). By using the formulation in \Cref{Thm: Rosen}-(\ref{item:algebraic3}), there exists a finite Galois extension $L/\QQ$ of degree $d=[L:\QQ]$, a map $\phi\colon \Gal(L/\QQ)\to L$, and a finite number of elements $b_1,\ldots,b_k\in L$ such that 
\begin{align}
\nonumber
\text{Im}(\phi)=\{b_{1},\ldots,b_{k}\}\ \text{and}\ a_{p}(E)\equiv \phi(\Frob_{\fp})=b_{i}\pmod \fp,
\end{align}
for some $i$ with $1\le i\le k$. 
By discarding $b_i$ if necessary, we may assume that $L$ is the Galois closure of $\QQ(b_1,\ldots,b_k)$. We set
\begin{align}
\nonumber
    S_{1}=\left\{ p\colon \text{good prime for}\ E\ |\ p\ \text{splits completely in}\ L\right\}.
\end{align}
Then the natural density is given as $\text{den}(S_{1})=d^{-1}>0$ by Chebotarev density theorem. Thus, there are infinitely many primes that belong to $S_{1}$.

We fix such $p$ and a prime ideal $\fp$ of $\OO_L$ lying over $p$. The condition $p\in S_{1}$ says $\Frob_{\fp}=\text{id}\in\Gal(L/\QQ)$. Moreover, we have $b:=\phi(\text{id})\in\QQ$ by definition of $g$. We can assume that $b\neq0$. Indeed, if $E$ is non-CM, then the natural density of the set of rational primes $p$ such that $a_{p}(E)=0$ is $0$ (cf. {\cite[Theorem B]{Elkies}}). In addition, the CM case follows from \Cref{lem: existence of primes}. In either case, there exist infinitely many primes $p\in S_1$ for which $a_{p}(E)\neq0$. Since $a_{p}(E)\in\ZZ$ and $b\in\frac{1}{N}\ZZ$ for some $N\in\ZZ_{\ge0}$, we have
\begin{align}\label{cong a_p=b}
    a_{p}(E)\equiv b\bmod p.
\end{align}
Using \eqref{cong a_p=b} and Hasse bound (cf. \cite[Chapter V, p.138, Theorem 1.1]{Sil}), we obtain
\begin{align}\label{eneq p}
    \frac{p}{N}\leq |a_{p}(E)-b| \leq 2\sqrt{p}+|b|.
\end{align}
\eqref{eneq p} says that $p$ lives in $p<K:=(N+\sqrt{N^{2}+N|b|})^{2}$. Note that $K$ is independent of the choice of $p$, so we have
\begin{align}\label{a_p=b}
    a_{p}(E)-b=0,\ {}^{\forall}p\in S_{1}\ \text{and}\ p>K.
\end{align}
We define the following sets by
\begin{align}
    \nonumber
    & S_{2}=\{ p\colon \text{good prime for}\ E\ |\ p\ \text{splits completely in}\ L,\ b\in\ZZ_{p},\ a_{p}(E)=b\ \text{for}\ p>K \},\\
\nonumber
  & S_{3}=\{ p\colon \text{good prime for}\ E\ |\ b\in\ZZ_{p},\ a_{p}(E)=b\ \text{for}\ p>K \}.
\end{align}
Then, we have $S_{2}\subseteq S_{3}$. 

\noindent
\underline{(i) $E$ is non-CM case} The equality
\begin{align}
    \nonumber
     \text{den}(S_{3})=\int_{\frac{\pi}{2}}^{\frac{\pi}{2}}\sin^{2}\theta \mathrm{d}\theta=0
\end{align}
holds by Sato--Tate conjecture (\Cref{Thm: Sato--Tate conj}-(\ref{item:Sato--Tate nonCM})). \\

\noindent
\underline{(ii) $E$ is CM case} By the condition $b\neq0$, we have the same equation
\begin{align}
\nonumber
    \text{den}(S_{3})=0
\end{align}
from \Cref{Thm: Sato--Tate conj}-(\ref{item:Sato--Tate CM}).

However, in both cases, it contradicts 
\begin{align}
\nonumber
     S_{2}\subseteq S_{3}\ \text{and}\ \text{den}(S_{2})=\text{den}(S_{1})>0.
\end{align}
Thus, we obtain the conclusion.
\end{proof}

\section{Transcendence for the Frobenius traces of abelian varieties}\label{section: fin. trans. abelian var. case}

In this section, we prove \Cref{Main 2} using the proof method in \cite{LZ1}. In what follows, we introduce the notation. Specifically, in \Cref{Main 2}, we consider the case $K=\QQ$. 

Let $A$ be an abelian variety of dimension $g$ over a number field $K$. Let $v$ be a fixed finite place on $K$. We denote by $K_v$ the completion of $K$ at $v$, and by $\OO_{K,v}$ its ring of integers. We say that $A$ has \textit{good reduction at $v$} if there exists an abelian scheme $A_{sch}$ over $\Spec(\OO_{K,v})$ such that its generic fiber $A_{sch} \times_{\Spec(\mathcal{O}_{K,v})} \Spec(K_v)$ is isomorphic to $A \times_{\Spec(K)} \Spec(K_v)$. If $A$ has good reduction at $v$, we call $v$ a \textit{good place} for $A$. For further details, we refer to \cite{ST} (cf. \Cref{remark: good reduction Galois representation}).

Let $\ell$ be a rational prime. For a positive integer $m$, there is a natural action of $G_K$ on the $\ell^m$-torsion points $A[\ell^m]$. Taking the inverse limit, we obtain the \textit{$\ell$-adic Galois representation}
\begin{align}
\nonumber
\rho_{A,\ell}\colon G_K\to\Aut_{\ZZ_\ell}(T_\ell A)\simeq\GL_{2g}(\ZZ_\ell),
\end{align}
where $T_\ell A$ is the $\ell$-adic Tate module of $A$. For any fixed polarization $\lambda\colon A\to A^{\vee}$, we obtain the Weil pairing on $T_\ell A$:
\begin{align}
\nonumber
e_{\ell,\lambda}\colon T_\ell A\times T_\ell A\to \ZZ_\ell(1),
\end{align}
where $\ZZ_\ell(1)$ is the Tate twist. For any prime $\ell$ such that $\ell\nmid\deg\lambda$, the pairing $e_{\ell,\lambda}$ is perfect and non-degenerate. It is known that $e_{\ell,\lambda}$ is Galois equivariant; that is, for any $\sigma\in G_K$ and any $P,Q\in T_\ell A$, we have
\begin{align}
\nonumber
e_{\ell,\lambda}({}^{\sigma}P,{}^{\sigma}Q)={}^{\sigma}(e_{\ell,\lambda}(P,Q)) = \chi_\ell(\sigma) e_{\ell,\lambda}(P,Q),
\end{align}
where $\chi_\ell \colon G_K \to \ZZ_\ell^\times$ is the $\ell$-adic cyclotomic character.
By fixing a symplectic basis for $T_\ell A$, the Galois equivariance implies that the image of $\rho_{A,\ell}$ is contained in the \textit{general symplectic group} $\GSp_{2g}(\ZZ_\ell)$, defined by
\begin{align}
\nonumber
\GSp_{2g}(\ZZ_\ell) = \left\{ X\in M_{2g}(\ZZ_\ell)\ \Big|\ {}^{\exists} \nu(X)\in\ZZ_\ell^{\times} \text{ s.t. } {}^{t}XJ_gX=\nu(X)J_g \right\}.
\end{align}
 It is known that, as an algebraic group, $\GSp_{2g}$ has dimension $2g^2+g+1$. The composition of $\rho_{A,\ell}$ with the similitude character $\nu \colon \GSp_{2g}(\ZZ_\ell)\to\ZZ_\ell^\times$ coincides with $\chi_\ell$. Taking the $\bmod\ \ell$ reduction, we obtain the residual representation:
\begin{align}
\nonumber
\overline{\rho}_{A,\ell}\colon G_K\to\GSp_{2g}(\F_\ell).
\end{align}
Similarly, for any positive integer $m>1$, we write the $\bmod\ \ell^m$ representation of $\rho_{A,\ell}$ by
\begin{align}
    \nonumber
    \overline{\rho}_{A,\ell^m}\colon G_K\to \GSp_{2g}(\ZZ/\ell^m\ZZ).
\end{align}

Note that the conditions in \Cref{Main 2} are imposed in order to apply the next theorem by Serre (\cite{Serre2}) as follows:

\begin{thm}[{\cite[TH\'EOR\`EME 3]{Serre2}}]\label{Thm: Serre open image}
     Let $A$ be an abelian variety over a number field $K$, such that
    \begin{enumerate}
        \item $\End_{\Kbar}(A)=\ZZ$,
        \item $\dim A =:g$ satisfies $g\equiv 1\pmod2$ or $g=2$ or $g=6$.
    \end{enumerate}
    Then, the $\ell$-adic Galois representation $\rho_{A,\ell}\colon G_K\to\GSp_{2g}(\ZZ_\ell)$ is surjective for all but finitely many $\ell$ not dividing the discriminant of the polarization of $A$.
\end{thm}
Before going to prove \Cref{Main 2}, we need the following lemmas. 
\begin{lem}\label{Hm}
Let $d=\dim\GSp_{2g}$. For an integer $m\ge 1$ and a prime number $\ell>2$, put $G_m=\GSp_{2g}(\ZZ/\ell^m \ZZ)$ and $H_m=\{a\in G_m\ \mid\ \tr(a)=0\}$. Then, $\#H_m=O(\ell^{md-1})$ for $\ell$ sufficiently large.
\end{lem}
\begin{proof}
Let $\pi_m:G_m\longrightarrow G_{m-1},\ M\mapsto M \mod \ell^{m-1}$. Put $K_m:={\rm Ker}(\pi_m)$. It is easy to see that $K_m$ 
is abelian and it is isomorphic to 
\begin{align}
\nonumber
\frak g=\mathfrak{gsp}_{2g}(\F_\ell)=\{X\in M_{2g}(\F_\ell)\ |\ {}^\exists c\in \F_\ell\ \text{such that }
{}^t X J_g+J_g X=cJ_g\}
\end{align}
by sending $I_{2g}+\ell^{m-1}Y$ to $X=\ell^{m-1}Y$ together with the identification  
$\ell^{m-1}M_{2g}(\ZZ/\ell^{m}\ZZ)\simeq M_{2g}(\F_\ell)$. Here, $J_g=\begin{pmatrix}
        0_g & I_g \\ -I_g & 0_g
    \end{pmatrix}\in M_{2g}(\ZZ/\ell^m\ZZ)$, $I_g$ and $0_g$ denote the identity matrix and the zero-matrix. Thus, $\# K_m=\ell^d$ holds. Let 
\begin{align}
\nonumber
H_mK_m=\{ab\mid a\in H_m,\ b\in K_m\}.
\end{align}
Since $K_m$ is a subgroup of $G_m$, the set $H_mK_m$ is stable under 
right multiplication by $K_m$. Moreover, every right $K_m$-coset contained 
in $H_mK_m$ has a representative in $H_m$. Indeed, if $x\in H_mK_m$, then 
$x=ab$ for some $a\in H_m$ and $b\in K_m$, and hence
\begin{align}
\nonumber
xK_m=abK_m=aK_m.
\end{align}
Thus, we may choose a complete system $I\subset H_m$ of representatives for 
the quotient set $H_mK_m/K_m$. We then have the disjoint union
\begin{align}
\nonumber
H_mK_m=\coprod_{a\in I}aK_m.
\end{align}
Since $H_m\subset H_mK_m$, intersecting both sides with $H_m$ gives
\begin{align}
\nonumber
H_m
=
\coprod_{a\in I}(aK_m\cap H_m).
\end{align}

For the set $aK_m\cap H_m$, 
pick an element $a(I_g+\ell^{m-1}Y)$ there. Since it also belongs to 
$H_m$, we have $0={\rm tr}(a(I_g+\ell^{m-1}Y))={\rm tr}(a)+
\ell^{m-1}{\rm tr}(aY)=\ell^{m-1}{\rm tr}(aY)$. Let $\overline{a}:=a\bmod \ell$. Then, the condition is equivalent to 
${\rm tr}(\overline{a}X)=0$ for $X\in M_{2g}(\F_\ell)$ corresponding to $\ell^{m-1}Y$ under $M_{2g}(\F_\ell)\simeq \ell^{m-1}M_{2g}(\ZZ/\ell^m\ZZ)$. Consider an $\F_\ell$-linear map $K_m\simeq \frak g 
\xrightarrow{X\mapsto  \tr(\bar{a}X)} \F_\ell$. Its kernel  has the cardinality $\ell^{d-1}$ if the trace map is non-trivial 
and $\ell^d$ otherwise. In either case, it yields 
$\#(aK_m\cap H_m )\le \ell^d$.  
Thus, we have $\# H_m\le \# I\cdot\ell^d$ and also 
$$\# H_m/\ell^d\le \# I=\# (H_mK_m/K_m) \le \# H_{m-1}$$
since $\pi_m$ is a group homomorphism and $\pi_m|_{H_mK_m/K_m}:H_m K_m/K_m\longrightarrow H_{m-1}$ is injective. 
Thus, inductively, we have 
$$\# H_m\le \ell^{(m-1)d}\# H_1.$$
Since $H_1$ is a closed algebraic variety over $\F_\ell$, we can apply the Lang--Weil theorem (cf. \cite[p.2729, THEOREM 2.3.2]{AA}), to conclude that $\# H_1=O(\ell^{d-1})$. Thus, 
$$\# H_m\le \ell^{(m-1)d}O(\ell^{d-1})=O(\ell^{md-1})$$
and the claim follows.
\end{proof}

\begin{lem}\label{lem: density of S}
    Let $A$ be an abelian variety over $\QQ$ of dimension $g$ satisfying the conditions in \Cref{Thm: Serre open image}. Put
    \begin{align}
    \nonumber
        & S=\{ p\ \colon \text{good prime for}\ A\ |\ a_p(A)=0 \}.
    \end{align}
    Then, we have $\mathrm{den}(S)=0$.
\end{lem}

\begin{proof}
    For a positive integer $m$, we put
    \begin{align}
    \nonumber
        S_m=\{ p\ \colon \text{good prime for}\ A\ |\ a_p(A)\equiv0\pmod{\ell^m} \}.
    \end{align}
    By definition, we see that $S\subset S_m$, and then, $\text{den}(S)\leq\text{den}(S_m)$ holds. Since $\GSp_{2g}$ is smooth over $\ZZ_\ell$, for each $m\ge 1$, the $\bmod\ \ell^m$ reduction $\GSp_{2g}(\ZZ_\ell)\to \GSp_{2g}(\ZZ_\ell/\ell^m \ZZ_\ell)$ is surjective. Thus,  
    the $\bmod\ \ell^{m}$ representation
    \begin{align}
        \nonumber
\overline\rho_{A,\ell^m}\colon G_\QQ\to\GSp_{2g}(\ZZ/\ell^{m}\ZZ)
    \end{align}
    is also surjective for any $m\geq 1$. 
    
    Let $K_{m}$ be the Galois extension which corresponds to $\Ker\overline\rho_{A,\ell^m}$. Then we have
    \begin{align}
        \nonumber
        \GSp_{2g}(\ZZ/\ell^{m}\ZZ)\simeq G_\QQ\Big/\Ker\overline{\rho}_{A,\ell^m}\simeq \Gal(K_{m}/\QQ).
    \end{align}
    Note that the first isomorphism holds due to the surjectivity of $\overline{\rho}_{A,\ell^m}$.
    
      Comparing the density, we have
    \begin{align}
        \nonumber
        \text{den}(S)\leq \text{den}(S_{m})=\frac{\#\{ \sigma\in\Gal(K_{m}/\QQ)\ |\ \tr(\overline{\rho}_{A,\ell^m}(\sigma))=0 \}}{\# \Gal(K_{m}/\QQ)}\leq\frac{\#H_{m}}{\#\GSp_{2g}(\ZZ/\ell^{m}\ZZ)}.
    \end{align}
    For any $\ell$ sufficiently large, the following estimate
    \begin{align}
        \nonumber
        \#\Gal(K_{m}/\QQ)=\#\GSp_{2g}(\ZZ/\ell^{m}\ZZ)\sim \ell^{md},\quad d:=\dim \GSp_{2g}=2g^2+g+1
    \end{align}
    holds by the well-known formula
    \begin{align}
    \nonumber
        \#\GSp_{2g}(\ZZ/\ell^m\ZZ)&=\ell^{(2m-1)g^2+(m-1)(g+1)}(\ell-1)\prod_{i=1}^{g}(\ell^{2i}-1)\\
    \nonumber
        &=\ell^{md}(1-\ell^{-1})\prod_{i=1}^{g}(1-\ell^{-2i}).
    \end{align}
    In addition, by \Cref{Hm}
    , we have
    \begin{align}
        \nonumber
        \# H_{m}=O(\ell^{md-1}).
    \end{align}
    Therefore, the following estimate
    \begin{align}
        \nonumber
        \text{den}(S)\leq\text{den}(S_{m})\ll\frac{1}{\ell},
    \end{align}
    holds for all but finitely many primes $\ell$. Therefore, we obtain the claim.
\end{proof}

\begin{rem}
    Although applying \Cref{Hm} with $m=1$ suffices for the proof of \Cref{lem: density of S}, we choose to state it separately as the statement itself is of interest in its own right.
\end{rem}

Before starting the proof of \Cref{Main 2}, we state two known lemmas.

\begin{lem}\label{lem: prime divisors}
    Let $n$ be a non-zero integer. For each $n$, we define
    \begin{align}
    \nonumber
        \omega(n)=\#\{ q\ |\ q\ \text{is a rational prime and}\ q\mid n \}.
    \end{align}
    Then, we have
    \begin{align}
    \nonumber
        \omega(n)=O(\log|n|).
    \end{align}
    In particular, if $|n|=O(X^c)$ for some constant $c>0$, then
    \begin{align}
    \nonumber
        \omega(n)=O(\log X).
    \end{align}
\end{lem}

\begin{proof}
    Let $q_1,\ldots,q_{\omega(n)}$ be the distinct prime divisors of $n$. Since each $q_j\ge2$ for any $j\in\{ 1,\ldots,\omega(n) \}$, we obtain
    \begin{align}
    \nonumber        2^{\omega(n)}\leq\prod_{j=1}^{\omega(n)}q_j\leq|n|.
    \end{align}
    Taking the logarithm, we have
    \begin{align}
        \nonumber
        \omega(n)\leq\frac{\log|n|}{\log2}\Rightarrow \omega(n)=O(\log|n|).
    \end{align}
    If $|n|=O(X^c)$ for some $c>0$, we see that $\log|n|=O(\log X)$, and then,
    \begin{align}
    \nonumber
        \omega(n)=O(\log X)
    \end{align}
    holds. 
\end{proof}

\begin{lem}\label{lem: counting lemma} 
Let $L$ be a number field of degree $d$, and let $b_1,\ldots, b_k \in L$. For each rational prime $p$, let $a_p$ be an integer. Suppose that
\begin{align}
    \nonumber
    |a_p|\leq Cp^\theta
\end{align}
for some constants $C>0$ and $0\leq\theta<1$.
Let $S$ be a set of rational primes satisfying the following conditions:
\begin{enumerate}
    \item[(i)] For every $p\in S$, there exist $i\in\{ 1,\ldots,k \}$ and a prime $\fp\mid p$ of $L$ such that $a_p\equiv b_i\pmod\fp$. 
    \item[(ii)] For every $p\in S$ and every $i\in\{ 1,\ldots,k \}$, one has $a_p\neq b_i$.
\end{enumerate}
Then, we have
\begin{align}
    \nonumber
    \#\{  p\leq X\ | \ p\in S \}=O(X^\theta\log X).
\end{align}
    In particular, $\mathrm{den}(S)=0$ holds. 
\end{lem}

\begin{proof}
    Since $b_i\in L$ for every $i\in\{ 1,\ldots,k \}$, there exists $N>0$ such that
    \begin{align}
    \nonumber
        Nb_i\in\OO_L,\quad {}^{\forall}i\in\{ 1,\ldots,k \}.
    \end{align}
    We may exclude the finitely many primes dividing $N$. 

    Let $p\in S$. By the first assumption, there exist $i\in\{ 1,\ldots,k \}$ and a prime $\fp\mid p$ of $L$ such that
    \begin{align}
    \nonumber
        a_p\equiv b_i\pmod\fp.
    \end{align}
    The above congruence says that
    \begin{align}
        \nonumber
        Na_p-Nb_i\in\fp.
    \end{align}
    Since $Na_p-Nb_i\in\OO_L$, we obtain
    \begin{align}
    \nonumber
        p\mid N_{L/\QQ}(Na_p-Nb_i).
    \end{align}
    By the second assumption, we see that
    \begin{align}
    \nonumber
        Na_p-Nb_i\neq0,
    \end{align}
    and hence, we have
    \begin{align}
    \nonumber
        N_{L/\QQ}(Na_p-Nb_i)\neq0.
    \end{align}
    Assume that $p\leq X$. Then, 
    \begin{align}
    \nonumber
        |N_{L/\QQ}(Na_p-Nb_i)|&=\prod_{\sigma\colon L\hookrightarrow\CC}|Na_p-N\sigma(b_i)|\\
    \nonumber
        &\leq \left( NCX^\theta + N\max_{\substack{1\leq i\leq k \\ \sigma\colon L\hookrightarrow \CC}}|\sigma(b_i)| \right)^d\\
        \nonumber
        &=O(X^{\theta d}).
    \end{align}
    By \Cref{lem: prime divisors}, we have
    \begin{align}
        \nonumber
        \omega(Na_p-N\sigma(b_i))=O(\log X).
    \end{align}
    Thus, for each fixed value $a_p$ and $i$, there are at most $O(\log X)$ possible rational primes $p\in S$ satisfying
    \begin{align}
        \nonumber
        p\mid Na_p-N\sigma(b_i).
    \end{align}

    On the other hand, the assumption
    \begin{align}
        \nonumber
        |a_p|\leq Cp^\theta
    \end{align}
    shows that there are $O(X^\theta)$ possible integer values $a_p$. Since $\{ b_1,\ldots,b_k \}$ is a finite set, we conclude that
    \begin{align}
        \nonumber
        \#\{  p\leq X\ |\ p\in S \}=O(X^\theta\log X).
    \end{align}
    Finally, since $0\leq\theta<1$, we have
    \begin{align}
    \nonumber
        X^\theta\log X=o\left(\frac{X}{\log X}\right).
    \end{align}
    By the prime number theorem, we obtain
    \begin{align}
        \nonumber
        \frac{\#\{  p\leq X\ |\  p\in S \}}{\#\{ p\leq X \}}\to0,\quad X\to\infty.
    \end{align}
    Thus, $\text{den}(S)=0$ holds.
\end{proof}

\begin{proof}[Proof of \Cref{Main 2}]
    Suppose that $\bm{a}(A)=(a_{p}(A)\bmod p)_{p}\in\PP_{\A}^{0}$. Let $L/\QQ$ and $\{ b_{1},\ldots,b_{k} \}\subset L$ be as in the proof of \Cref{Main 1}. 
    
    Let us choose a positive integer $N$ such that $Nb_i\in\OO_L$ for every $i\in\{ 1,\ldots,k \}$. We take a sufficiently large prime $\ell$ such that
    \begin{enumerate}
        \item The Galois representation $\rho_{A,\ell}\colon G_\QQ\to\GSp_{2g}(\ZZ_\ell)$ is surjective,
        \item $\ell\nmid N$,
        \item  $\ell$ and $b_i$ are coprime; that is, $\ell\nmid N_{L/\QQ}(Nb_i)$ for every $i\in\{ 1,\ldots,k \}$ such that $b_i\neq0$.
    \end{enumerate}

    For such $\ell$, define
    \begin{align}
    \nonumber
        S_\ell=\{ p\ \colon \text{good prime for}\ A\ |\ a_p(A)\neq0,\ \ell\mid a_p(A)  \}.
    \end{align}
    By \Cref{lem: density of S}, the set of rational primes that $A$ has good reduction such that $a_p(A)=0$ has natural density $0$. Thus, the natural density of $S_\ell$ is the same as one of
    \begin{align}
        \nonumber
        S^\prime_\ell=\{ p\ \colon \text{good prime for}\ A\ |\ \tr(\overline{\rho}_{A,\ell}(\Frob_p))=0 \}.
    \end{align}
    By Chebotarev density theorem (cf. {\cite[Equation 9, Section 2.1]{Serre Chebotarev}}), we have
    \begin{align}
        \nonumber
        \#\{  p\leq X\ |\ p\in S_\ell \}=C\cdot\frac{X}{\log X}+o\left( \frac{X}{\log X} \right)
    \end{align}
    for some positive constant $C$. In particular, 
    \begin{align}
        \label{estimation}
        \#\{ p\leq X\ |\  p\in S_\ell \}\gg_\ell\frac{X}{\log X}.
    \end{align}

    Now we claim that $a_p(A)\neq b_i$ for every $p\in S_\ell$ and every $i\in\{ 1,\ldots,k \}$. If $b_i=0$, this follows from the definition of $S_\ell$. Thus, we may assume that $b_i\neq0$. If $a_p(A)=b_i$, then $b_i\in\ZZ$ holds. In addition, the condition $\ell\mid a_p(A)$ implies that
    \begin{align}
        \nonumber
        \ell\mid N_{L/\QQ}(Nb_i),
    \end{align}
    contrary to the choice of $\ell$.

    Moreover, the Weil bound gives
    \begin{align}
        \nonumber
        |a_p(A)|\leq2g\sqrt{p}.
    \end{align}
    Thus, we have
    \begin{align}
        \nonumber
        |N_{L/\QQ}(Na_{p}(A)-Nb_{i})|&\leq \left(2gN\sqrt{X}+N\max_{\substack{1\leq i\leq k \\ \sigma\colon L\hookrightarrow\CC}}|\sigma(b_i)|\right)^{[L:\QQ]}\\
        \nonumber
        &=O(X^{[L:\QQ]/2}).
    \end{align}
    Applying \Cref{lem: counting lemma} with $\theta=1/2$, we obtain
    \begin{align}
        \label{contradiction estimate}
        \#\{  p\leq X\ |\ p\in S_\ell \}=O(\sqrt{X}\log X).
    \end{align}
    However, \eqref{contradiction estimate} contradicts \eqref{estimation}. Therefore, $\bm{a}(A)\notin\PP_\A^{0}$.
\end{proof}



\begin{thebibliography}{99}

\bibitem{AA}
A. Aizenbud, and N. Avni,
\textit{Counting points of schemes over finite rings and counting representations of arithmetic lattices},
Duke Math. J., \textbf{167} (2018), no. 14, 2721--2743.

\bibitem{AF}
T. Anzawa, and H. Funakura,
\textit{Congruences of the $q$-Fibonacci sequence related with its transcendence},
Ramanujan J., \textbf{63}, No.4 (2024) 1057--1072.

\bibitem{BGHT}
T.~Barnet-Lamb, D.~Geraghty, M.~Harris, and R.~Taylor, 
\textit{A family of Calabi-Yau varieties and potential automorphy II.},
Publ. Res. Inst. Math. Sci., \textbf{47} (2011) no. 1, 29–98.

\bibitem{Del}
P. Deligne,
\textit{La conjecture de Weil. $\mathrm{I}$.},
Inst. Hautes Etudes Sci. Publ. Math., \textbf{43} (1974) 273--307.

\bibitem{Elkies}
N. D. Elkies,
\textit{Distribution of supersingular primes},
Journ\'ees arithm\'etiques (Luminy, 1989),
Ast\'erisque \textbf{198--199--200} (1991), 127--132.


\bibitem{LZ2}
F. Luca, and W. Zudilin,
\textit{Irrationality and transcendence questions
in the ‘poor man’s adèle ring’},
Ramanujan J., \textbf{67}, No.88 (2025). 

\bibitem{LZ1}
F. Luca, and W. Zudilin,
\textit{Poor man's transcendence for Frobenius traces of elliptic curves},
to appear in Advanced Studies in Pure Math., arXiv: 2507.14773.

\bibitem{MS}
T. Matsusaka, and S. Seki,
\textit{Some results on naive transcendence in the ring of integers modulo infinitely large primes},
 to appear in Journal of the Australian Mathematical Society, arXiv:2604.25566.

\bibitem{Murty}
V-K.~Murty,
\textit{On the Sato-Tate conjecture.},
Number theory related to Fermat's last theorem (Cambridge, Mass., 1981), pp. 195–205, Progr. Math., 26, Birkhäuser, Boston, MA, 1982.

\bibitem{Rosen2}
J. Rosen, 
\textit{A choice-free absolute Galois group and Artin motives},
preprint, arXiv: 1706.06573.

\bibitem{Rosen}
J. Rosen,
\textit{A finite analogue of the ring of algebraic numbers},
J. Number Theory, \textbf{208} (2020) 59--71.

\bibitem{Serre}
J.-P. Serre,
\textit{Propriétés galoisiennes des points d'ordre fini des courbes elliptiques},
Invent. Math., \textbf{15} (1971) 259--331.

\bibitem{Serre Chebotarev}
J.-P. Serre,
\textit{Quelques applications du th\'eor\`eme de densit\'e de Chebotarev},
Publications Mathématiques de l'IHES, \textbf{54} (1981), 123-201.

\bibitem{Serre2}
J.-P. Serre,
\textit{R\'esum\'e des cours de 1985--1986},
Annuaire du Coll\'ege de France (1986), 95--99.

\bibitem{ST}
J.-P. Serre, and J. Tate,
\textit{Good reduction of abelian varieties},
Ann.Math, \textbf{88} (1968) 492--517.

\bibitem{Sil}
J-H.~Silverman, 
\textit{The arithmetic of elliptic curves},
Second edition. Graduate Texts in Mathematics, 106. Springer, Dordrecht, 2009. xx+513 pp.

\bibitem{Zaytsev}
Alexey Zaytsev,
\textit{Generalization of Deuring reduction theorem},
J. Algebra, \textbf{392} (2013) 97--114.

\end{thebibliography}
\end{document}